\documentclass[11pt, reqno]{amsart}
\usepackage{amsrefs, amsmath, amssymb, amsthm, graphicx,marvosym,pifont,cancel}
\usepackage{mathtools}
\usepackage{color}
\usepackage{cite}

\def\veps{\varepsilon}
\def\vp{\varphi}

\def\eq#1{(\ref{#1})}

\def\({\left(\begin{array}{cccccc}}
\def\){\end{array}\right)}

\def\eq#1{(\ref{#1})}

\def\({\left(\begin{array}{cccccc}}
\def\){\end{array}\right)}

\def\bes{\begin{eqnarray}}
\def\ees{\end{eqnarray}}

\newcommand{\vertiii}[1]{{\left\vert\kern-0.25ex\left\vert\kern-0.25ex\left\vert #1 
    \right\vert\kern-0.25ex\right\vert\kern-0.25ex\right\vert}}

\newcommand{\beq}{\begin{equation}}
\newcommand{\eeq}{\end{equation}}
\newcommand{\bea}{\begin{eqnarray}}
\newcommand{\eea}{\end{eqnarray}}
\newcommand{\beann}{\begin{eqnarray*}}
\newcommand{\eeann}{\end{eqnarray*}}

\newcommand{\RR}{\mathbb{R}}

\newcommand{\NN}{\mathbb{N}}

\newcommand{\bv}{BV}
\newcommand{\bpv}{BPV}

\newcommand{\epsp}{(\veps,p)\text{-}\mathrm{Var\,}}
\newcommand{\epsdeltaone}{(\veps+\delta,1)\text{-}\mathrm{Var\,}}
\newcommand{\epsnaughtp}{(\veps_0,p)\text{-}\mathrm{Var\,}}
\newcommand{\epsnaughtone}{(\veps_0,1)\text{-}\mathrm{Var\,}}
\newcommand{\epsnaughtinf}{(\veps_0,\infty)\text{-}\mathrm{Var\,}}
\newcommand{\epsone}{(\veps,1)\text{-}\mathrm{Var\,}}
\newcommand{\epsinf}{(\veps,\infty)\text{-}\mathrm{Var\,}}

\newcommand{\bp}{\begin{proof}}
\newcommand{\ep}{\end{proof}}

\DeclareMathOperator{\dv}{div}

\DeclareMathOperator{\ess}{ess\, sup}

\DeclareMathOperator{\var}{var}
\DeclareMathOperator{\evar}{\veps-var}

\DeclareMathOperator{\essen}{ess}

\DeclareMathOperator{\V}{Var}

\newtheorem{theorem}{Theorem}[section]
\newtheorem{proposition}[theorem]{Proposition}

\newtheorem{observation}[theorem]{Observation}
\newtheorem{lemma}[theorem]{Lemma}
\newtheorem{definition}[theorem]{Definition}

\newtheorem{example}[theorem]{Example}
\newtheorem{remark}[theorem]{Remark}

\numberwithin{equation}{section}

\begin{document}

\title{Extensions of BV compactness criteria}

\author{Helge Kristian Jenssen} 
\address{ H.\ K.\ Jenssen, Department of Mathematics, Penn State University, University Park, 
State College, PA 16802, USA ({\tt jenssen@math.psu.edu}).}

\thanks{This work was partially supported by the National Science Foundation [grant DMS-1813283].}

\date{\today}
\begin{abstract} 
	Helly's selection theorem provides a criterion for compactness of sets of 
	single-variable functions with bounded pointwise variation. Fra{\v{n}}kov{\'a} has 
	given a proper extension of Helly's theorem to the setting of single-variable regulated 
	functions.
	We show how a similar approach yields extensions of the standard compactness 
	criterion for multi-variable functions of bounded variation. 
\end{abstract}

\maketitle



\tableofcontents

\section{Introduction}\label{intro}
Helly's selection theorem (Theorem \ref{helly} below) provides a 
compactness criterion for sequences of one-variable functions 
with uniformly bounded pointwise variation.
In her work on regulated functions, i.e., one-variable functions
admitting finite left and right limits at all points, Fra{\v{n}}kov{\'a}
\cite{fr} provided an extension of Helly's theorem. Simple examples 
demonstrate that Fra{\v{n}}kov{\'a}'s
theorem (Theorem \ref{fr_thm} below) provides a genuine extension. 
In particular, it guarantees pointwise everywhere 
convergence of a subsequence in some cases where the  
sequence is not bounded in variation. 

The main objective of the present work is to provide a generalization of Fra{\v{n}}kov{\'a}'s
theorem to the multi-variable setting. For a ``nice'' open set $\Omega\subset\RR^N$,
with $N\geq 1$, the space $\bv(\Omega)$ of functions of bounded variation admits a 
compactness result \`a la Helly's (see Theorem \ref{multi_var_compact} below): 
any sequence which is bounded in $BV$-norm
contains a subsequence converging in $L^1$-norm to a $BV$-function. 
We shall see how this compactness criterion can play the same role that Helly's 
theorem plays in Fra{\v{n}}kov{\'a}'s approach in \cite{fr}. As a consequence we obtain 
a compactness result (Theorem \ref{multi_var_frankova_any_p}) 
which guarantees $L^1$-convergence of a subsequence 
in certain cases where the original sequence is unbounded in $\bv$.

A key ingredient in Fra{\v{n}}kov{\'a}'s work \cite{fr} is the notion of {\em $\veps$-variation} 
of a bounded function $u$ of one variable: for $\veps>0$, $\evar u$ is defined as the smallest amount
of pointwise variation a function uniformly $\veps$-close to $u$ can have.
We shall introduce generalizations of this notion to functions
of any number of variables (including 1-variable functions). For the purpose
of extending the multi-d criterion for $BV$-compactness, it is natural to define 
$\veps$-variation for general $L^1$-functions. In addition, we want to consider more general
norms than the uniform norm in measuring distance between a function 
and its $BV$-approximants. The standard $L^p$-norms, with $1\leq p\leq\infty$, provide 
a natural choice, and we therefore introduce the notion of {\em $(\veps,p)$-variation} of  
general $L^1$-functions (Definition \ref{eps_p_varn_defn}). 
Our main result is that the $BV$-compactness
criterion admits an extension to this setting (Theorem \ref{multi_var_frankova_any_p}).
Just as for Fra{\v{n}}kov{\'a}'s extension of Helly's theorem, simple examples
demonstrate that this is a proper extension.

We note that even in the case of one-variable functions and with $p=\infty$, which is the setting 
closer to that of Fra{\v{n}}kov{\'a}'s, our result is not identical to hers.
The latter concerns pointwise variation and everywhere pointwise convergence, i.e., 
the setting of Helly's theorem. In contrast, our Theorem \ref{multi_var_frankova_any_p}
is formulated in terms of variation and $L^p$-convergence of (strictly speaking) equivalence
classes of functions agreeing up to null-sets. (For the relation between 
pointwise variation and variation of a one-variable function, see Remark \ref{notns_of_varn}.)
In particular, for $p=\infty$, we employ $L^\infty$-norm rather than uniform norm.
This distinction is of relevance when we seek to generalize the notion of regulated function to 
the multi-variable case (see below).

\begin{remark}
While our main result applies with any 
choice of $L^p$-norm, the values $p=1$ and $p=\infty$ 
are the more relevant ones. The case $p=1$ is natural since the original
criterion for $BV$-compactness in several dimensions guarantees $L^1$-convergence
of a subsequence. On the other hand, $p=\infty$ provides a setting closer to 
that of Fra{\v{n}}kov{\'a} \cite{fr}. Also, we have been able to establish 
certain useful properties only when $p=1$ or $p=\infty$. These concern attainment of 
$(\veps,p)$-variation, continuity with respect to $\veps$, and lower semi-continuity 
with respect to $L^1$-convergence; see Sections 
\ref{multi_var_case_p=1}-\ref{multi_var_case_p=infty}.
\end{remark}

It turns out that, besides providing a means for extending Helly's theorem, Fra{\v{n}}kov{\'a}'s 
notion of $\veps$-variation also yields a characterization of regulated functions. 
Indeed, a function is regulated if and only if its 
$\veps$-variation is finite for each $\veps>0$, \cite{fr}. 
Neither the definition of regulated functions (Definition \ref{reg}), 
nor various characterizations (see Theorem 2.1, pp.\ 213-214 in \cite{dn}), 
admit a straightforward generalization to higher dimensions (see also \cite{da}). 
In contrast, Fra{\v{n}}kov{\'a}'s characterization may appear to offer an obvious 
way of defining regulated functions of several variables. However, there is a catch:
$\veps$-variation for a one-variable function is defined in terms of pointwise 
variation, a notion lacking in higher dimensions. For functions of several 
variables one could of course replace it by variation; however, the resulting definition,
when restricted to the one-variable case,  
will not reproduce the class of regulated functions of one variable.

Instead, having introduced the notion of $(\veps,p)$-variation, we use it to
define {\em $p$-regulated} functions as those functions whose 
$(\veps,p)$-variation is finite for all $\veps>0$. The resulting function space 
is introduced in Section \ref{eps_p_varns_p_reg_fncs}, and this provides the setting 
for our main result on extensions of the standard $\bv$ compactness criterion for 
multi-variable functions. 
For the particular case of $\infty$-regulated functions of one variable,
we show that these are precisely the functions that are ``essentially regulated,''
i.e., possessing a regulated version (see Proposition \ref{1_d_reg_vs_infty_reg}).

Finally, we comment on the possibility of applying Fra{\v{n}}kov{\'a}'s strategy to other 
compactness criteria in function spaces. Such criteria invariably involve a requirement that 
the given sequence (or set) of functions satisfy some uniform requirement. In essence, Fra{\v{n}}kov{\'a}'s 
strategy amounts to replacing a uniform requirement on the given functions by a (weaker)
uniform requirement on nearby functions. Specifically, in Fra{\v{n}}kov{\'a}'s 
extension of Helly's theorem, this is done by replacing the requirement of a 
uniform variation bound by the weaker requirement of uniformly bounded $\veps$-variations 
(see Definitions \ref{unif_eps_varns} and  \ref{multi_var_unif_eps_varns}).

It is natural to ask if a similar approach can be applied to extend other compactness theorems
for function spaces, including Fra{\v{n}}kov{\'a}'s own theorem.
This is a somewhat open-ended question as there is freedom in how to set up an extension.
However, we have not been able to extend Fra{\v{n}}kov{\'a}'s theorem, the Ascoli-Arzel\`a theorem, 
or the Kolmogorov-Riesz  theorem. In each case we find that a natural implementation 
of Fra{\v{n}}kov{\'a}'s strategy fails: the assumptions on the given function sequence are so strong 
that the original compactness criterion directly applies to it. One might say that 
these compactness results are saturated with respect to Fra{\v{n}}kov{\'a}'s strategy.

The rest of the article is organized as follows. Notation and conventions are 
recorded below. Section \ref{frankova_helly_extn}  
recalls Helly's theorem for sequences of one-variable functions with bounded (pointwise) 
variation, together with Fra{\v{n}}kov{\'a}'s extension to the space of regulated functions.
We include the key definitions introduced in \cite{fr} and recall without proof some 
of the results from \cite{fr}, including the characterization of regulated functions in 
terms of $\veps$-variation. For completeness we include a proof of Fra{\v{n}}kov{\'a}'s
theorem. Section \ref{multi_var_case} concerns the generalization
of Fra{\v{n}}kov{\'a}'s strategy to the case of functions of several variables. 
We first recall the standard compactness result for bounded sequences in $\bv(\Omega)$,
and then define $(\veps,p)$-variation of $L^1(\Omega)$-functions. 
The space $\mathcal R_p(\Omega)$ of $p$-regulated functions is introduced in Definition \ref{p_reg}. 
The main result, Theorem \ref{multi_var_frankova_any_p}, is then formulated and 
proved by mimicking Fra{\v{n}}kov{\'a}'s proof. Further properties of $(\veps,p)$-variation 
for $p=1$ and $p=\infty$ are established in Sections 
\ref{multi_var_case_p=1}-\ref{multi_var_case_p=infty}. Section \ref{reg_vs_infty_reg_1_d}
provides the relationship between $\infty$-regulated and standard regulated functions
of one variable.
Finally, in Section \ref{aa} we describe our negative findings about the possibility of 
applying Fra{\v{n}}kov{\'a}'s strategy to other compactness criteria.

\medskip

\noindent{\bf Notations and conventions:} For sequences of functions we write $(u_n)$ or $(u_n)_n$ for $(u_n)_{n=1}^\infty$.
Given a set of functions $X$, we write $(u_n)\subset X$ to mean that $u_n\in X$ for all $n\geq 1$.
We write $(u_{n(k)})\subset (u_n)$ to mean that $(u_{n(k)})_k$ is a subsequence of $(u_n)_n$.
For $m\in\NN$ we fix a norm $|\cdot|$ on $\RR^m$.
For any set $U$ and any function $u:U\to\RR^m$ we define its uniform norm by
\[\|u\|:=\sup_{x\in U} |u(x)|.\]
The set of bounded functions on $U$ is denoted 
\[\mathcal B(U):=\{u:U\to\RR^m\,|\, \|u\|<\infty\};\]
it is a standard result that $(\mathcal B(U),\|\cdot\|)$ is a Banach space.

For a normed space $(\mathcal X,\vertiii{\cdot})$, a sequence $(u_n)\subset \mathcal X$ 
is  {\em bounded} provided $\sup_n\vertiii{u_n}<\infty$. For a Lebesgue measurable set 
$\Omega\subset\RR^N$ ($N\geq 1$), $|\Omega|$ denotes its Lebesgue measure.
The open ball of radius $r$ about $x\in\RR^N$ is denoted $B_r(x)$.
For an open set $\Omega\subset\RR^N$, $L^p(\Omega)$, $1\leq p\leq \infty$, denotes the 
set of Lebesgue measurable functions with finite $L^p(\Omega)$-norm
\[\|u\|_{L^p(\Omega)}\equiv\|u\|_{p}=\Big(\int_\Omega |u(x)|^p\, dx\Big)^\frac{1}{p}\qquad \text{for $1\leq p< \infty$,}\]
and 
\[\|u\|_{L^\infty(\Omega)}\equiv\|u\|_\infty=\ess_{x\in\Omega} |u(x)|. \]
Throughout, the terms ``measurable'', ``almost everywhere'' and ``for almost all'' are 
understood with respect to Lebesgue measure. A {\em version} of a measurable function 
$u:\Omega\to\RR$ refers to any function $\bar u:\Omega\to\RR$
agreeing almost everywhere with $u$.

We use the notation ``$\var u$'' for the {\em pointwise variation} of a one-variable function 
$u$, and the notation ``$\V u$'' for the {\em variation} of a (one- or multi-variable) function $u$. These are
defined in \eq{var} and \eq{Var}, respectively. Our notation differs from that of \cite{afp} and \cite{le}
which use pV$(u,\Omega)$ and $\V u$, respectively, for 
the pointwise variation of a one-variable function, and $V(u,\Omega)$ and $|Du|(\Omega)$, respectively,
for the variation. (We avoid the notation pV$(u,\Omega)$ since we shall later define the notions of 
$(\veps,p)$-variation, where $p$ denotes the exponent of an $L^p$-space.) 
Remark \ref{notns_of_varn} recalls the relationship between pointwise variation and variation
of one-variable functions. 

Finally, for two sequences of real numbers $(A_n)$ and $(B_n)$, 
we write $A_n\lesssim B_n$ to mean that there is a finite number $C$, independent of $n$,
such that $A_n\leq CB_n$ for all $n\geq 1$. $A_n\sim B_n$ means that both $A_n\lesssim B_n$ 
and $B_n\lesssim A_n$ hold.

\section{Fra{\v{n}}kov{\'a}'s extension of Helly's theorem}\label{frankova_helly_extn}
In this section we fix an open and bounded interval $I=(a,b)\subset\RR$.
A function $u:I\to\RR^m$ is of {\em bounded pointwise variation} provided
\beq\label{var}
	\var u:=\sup\, \sum_{i=1}^k |u(x_i)-u(x_{i-1})|<\infty,
\eeq
where the supremum is over all $k\in\NN$ and all finite selections of points 
$x_0<x_1<\cdots<x_k$ in $I$. We follow \cite{le} and denote the set of such functions by
$\bpv(I)$. We recall Helly's selection theorem (see e.g.\ \cite{le}; recall that $\|\cdot\|$ 
denotes uniform norm):
\begin{theorem}[Helly]\label{helly}
	Assume $(u_n)\subset \mathcal B(I)$ satsfies $\sup_n\|u_n\|<\infty$ and 
	$\sup_n\var u_n<\infty$. Then there is a subsequence $(u_{n(k)})\subset(u_n)$ 
	and a function $u\in \bpv(I)$ such that $u(x)=\lim_{k} u_{n(k)}(x)$ for every $x\in I$.
	Furthermore, $\var u\leq\liminf_{k}\var u_{n(k)}$.
\end{theorem}
Next we describe Fra{\v{n}}kov{\'a}'s
extension of Helly's theorem. This requires a few definitions.

\begin{definition}\label{reg}
	A function $u\in\mathcal B(I)$ is {\em regulated}  
	provided its right and left 
	limits exist (as finite numbers) at all points of $I$, it has a finite 
	right limit at the left endpoint, and a finite left limit at the right endpoint. The class of 
	regulated functions on $I$ is denoted $\mathcal R(I)$. 
\end{definition}
\begin{remark}
	We have opted to work on a bounded and open interval $I$ so that the setting 
	is a special case of the multi-variable setting in Section \ref{multi_var_case}.
	In contrast, Fra{\v{n}}kov{\'a} \cite{fr} considers regulated functions on closed 
	intervals $[a,b]$. However, since Definition \ref{reg} requires finite one-sided
	limits at the endpoints, any regulated function in the sense above extends trivially 
	to a regulated function on $[a,b]$ in the sense of \cite{fr}. 
\end{remark}
It is immediate that $\mathcal R(I)$ is a proper subspace of $\mathcal B(I)$.
Fra{\v{n}}kov{\'a} \cite{fr} introduced the following definitions.
\begin{definition}\label{epsilon_varn}
	For $u\in \mathcal B(I)$ and  $\veps>0$, we define the {\em $\veps$-variation} 
	of $u$ by
	\beq\label{eps_varn}
		\evar u:=\inf_{v\in \mathcal U(u;\veps)} \var v,
	\eeq
	where 
	\beq\label{V}
		\mathcal U(u;\veps):=\{v\in\bpv(I)\,|\, \|u-v\|\leq \veps\},
	\eeq
	with the convention that infimum over the empty set is $\infty$. 
\end{definition}
\begin{definition}\label{unif_eps_varns}
	A set of functions $\mathcal F\subset \mathcal R(I)$ has {\em uniformly
	bounded $\veps$-variations} provided
	\[\sup_{u\in\mathcal F}\,\,\evar u<\infty\qquad\text{for every $\veps>0$}.\]
\end{definition}
For $u\in\mathcal B(I)$,  let $J(u)$ denote the jump set of $u$, i.e.,
\[J(u):=\{x\in I\,|\, \text{at least one of $u(x+)$ or $u(x-)$ differs from $u(x)$}\}.\]
(Here $u(x\pm)$ denotes $\lim_{y\to x\pm}u(y)$, respectively.)
A function $u\in \mathcal B(I)$ is a {\em step function} provided there is a finite, increasing 
sequence of points $x_0=a<x_1<\dots<x_{m-1}<x_m=b$ such that $u$ is constant on 
each of the open intervals $(x_i,x_{i+1})$, $i=0,\dots,m$.
The following results about regulated functions are known (for proofs see \cite{di,fr,dn}; (R3) and 
(R4) are Propositions 3.4 and 3.6 in \cite{fr}, respectively):
\medskip 

\begin{enumerate}
	\item[(R1)] If $u\in\mathcal R(I)$, then $J(u)$ is countable.\\
	\item[(R2)] For $u\in \mathcal B(I)$, $u\in\mathcal R(I)$ if and only if 
	it is the uniform limit of step functions on $I$. It follows that 
	$(\mathcal R(I),\|\cdot\|)$ is a proper, closed subspace of $(\mathcal B(I),\|\cdot\|)$.\\
	\item[(R3)] For $u\in \mathcal B(I)$, $u\in\mathcal R(I)$ if and only if 
	$\evar u<\infty$ for every $\veps>0$.\\
	\item[(R4)] Assume $(u_n)\subset \mathcal R(I)$ and $u_n(x)\to u(x)$ for every $x\in I$.
	Then 
	\[\evar u\leq\liminf_n\, \evar u_n\qquad\text{for every $\veps>0$.}\]
	If, in addition, $(u_n)$ has uniformly bounded $\veps$-variations, then
	$u\in\mathcal R(I)$.
\end{enumerate}
\medskip 

\noindent Note that (R3) provides a characterization of regulated functions. 
In Section \ref{multi_var_case} we will use this to motivate the 
notion of $p$-regulated functions of any number of variables. 

Before stating and proving Fra{\v{n}}kov{\'a}'s extension of Helly's theorem, we 
make the following observation.
\begin{observation}\label{obs1}
        Assume $(z_n)\subset\mathcal R(I)$ is bounded and has uniformly bounded $\veps$-variations.
        Then, for each $\veps>0$ there is a finite number $K_\veps$ and a sequence $(z_n^\veps)\subset\bpv(I)$ 
        satisfying
        \[\var z_n^\veps\leq K_\veps,\qquad \|z_n-z_n^\veps\|\leq\veps \qquad\text{for all $n\geq1$}.\]
        It follows that $(z_n^\veps)$ satisfies the assumptions in Helly's theorem; there is 
        therefore a subsequence $(z_{n(k)}^\veps)\subset(z_n^\veps)$ and a $z^\veps\in\bpv(I)$ so that 
        $z^\veps(x)=\lim_k z_{n(k)}^\veps(x)$ for every $x\in I$.
\end{observation}
Fra{\v{n}}kov{\'a}'s extension of Helly's theorem is the following result:
\begin{theorem}[Fra{\v{n}}kov{\'a}]\label{fr_thm}
	Assume $(u_n)\subset\mathcal R(I)$ is bounded and has uniformly bounded 
	$\veps$-variations. Then there is a subsequence $(u_{n(k)})\subset(u_n)$
	and a function $u\in\mathcal R(I)$ such that 
	$u(x)=\lim_k u_{n(k)}(x)$ for every $x\in I$.
\end{theorem}
\begin{remark}\label{f_thm_rmk}
        This is Theorem 3.8 in \cite{fr}, for which Fra{\v{n}}kov{\'a} provided two different proofs. 
        The first of these (outlined on pp.\ 48-49 in \cite{fr}) is relevant to us, and we 
        therefore provide the details of the argument. We also note that Fra{\v{n}}kov{\'a}'s 
        theorem provides a genuine extension of Helly's theorem. I.e., it specializes to 
        Helly's theorem when the conditions in Theorem \ref{helly} are met, and
        it also provides convergence of a subsequence in cases where the original sequence $(u_n)$ is 
        unbounded in $\bpv(I)$; see Example \ref{ex} and Example \ref{2nd_ex}.
\end{remark}

\medskip\noindent
{\it Proof of Theorem \ref{fr_thm}.}  Fix a strictly decreasing sequence $(\veps_l)$
converging to zero. We shall use Observation \ref{obs1} repeatedly, and then employ a 
diagonal argument.

For $l=1$ apply Observation \ref{obs1} to the original sequence $(u_n)$ with $\veps=\veps_1$ 
to get a sequence-subsequence pair $(v_n^{\veps_1})\supset(v_{n_1(k)}^{\veps_1})$ in $\bpv(I)$
and a $v^1\in\bpv(I)$ satisfying 
\[\|u_{n_1(k)}-v_{n_1(k)}^{\veps_1}\|\leq \veps_1\quad \text{for all $k\geq 1$, and}\qquad
v^1(x)=\lim_k v_{n_1(k)}^{\veps_1}(x)\quad \text{for every $x\in I$.}\]
For $l=2$ apply Observation \ref{obs1} to the sequence $(u_{n_1(k)})$ with $\veps=\veps_2$ 
to get a sequence-subsequence pair $(v_{n_1(k)}^{\veps_2})\supset(v_{n_2(k)}^{\veps_2})$ 
in $\bpv(I)$ and a $v^2\in\bpv(I)$ satisfying 
\[\|u_{n_2(k)}-v_{n_2(k)}^{\veps_2}\|\leq \veps_2\quad \text{for all $k\geq 1$, and}\qquad
v^2(x)=\lim_k v_{n_2(k)}^{\veps_2}(x)\quad \text{for every $x\in I$.}\]
Continuing in this manner we obtain for each index $l$ a sequence-subsequence pair
$(v_{n_{l-1}(k)}^{\veps_l})\supset(v_{n_l(k)}^{\veps_l})$ in $\bpv(I)$ and a $v^l\in\bpv(I)$ satisfying
\[\|u_{n_l(k)}-v_{n_l(k)}^{\veps_l}\|\leq\veps_l\quad \text{for all $k\geq 1$, and}\qquad
v^l(x)=\lim_k v_{n_l(k)}^{\veps_l}(x)\quad \text{for every $x\in I$.}\]
Next, for $l\geq 1$ fixed, consider the diagonal index sequence $(n_k(k))_{k\geq l}$,
which is a subsequence of $(n_l(j))_{j\geq 1}$. With $ n(k):=n_k(k)$ we therefore get
that: for each $l\geq 1$, there holds
\beq\label{key1}
	\|u_{ n(k)}-v_{ n(k)}^{\veps_l}\|\leq\veps_l\quad \text{for all $k\geq l$, and}
	\quad \lim_k v_{ n(k)}^{\veps_l}(x)= v^l(x)\quad \text{for every $x\in I$.}
\eeq
We claim that $(v^l)$ is a Cauchy sequence in $(\mathcal B(I),\|\cdot\|)$. Indeed, given $\delta>0$ 
we first choose an index $l(\delta)$ so that $\veps_l\leq\frac{\delta}{2}$ for $l\geq l(\delta)$. Then,
for any $x\in I$, if $l,q\geq l(\delta)$ and $k\geq \max(l,q)$ we get from \eq{key1}${}_1$ that
\begin{align*}
	|v^l(x)-v^q(x)|
	&\leq |v^l(x)-v_{n(k)}^{\veps_l}(x)|+|v_{n(k)}^{\veps_l}(x)-u_{n(k)}(x)|\\
	&\quad+|u_{n(k)}(x)-v_{n(k)}^{\veps_q}(x)|+|v_{n(k)}^{\veps_q}(x)-v^q(x)|\\
	&\leq |v^l(x)-v_{ n(k)}^{\veps_l}(x)|+\|v_{n(k)}^{\veps_l}-u_{n(k)}\|\\
	&\quad+\|u_{n(k)}-v_{n(k)}^{\veps_q}\|+|v_{n(k)}^{\veps_q}(x)-v^q(x)|\\
	&\leq |v^l(x)-v_{n(k)}^{\veps_l}(x)|+\veps_l+\veps_q+|v_{n(k)}^{\veps_q}(x)-v^q(x)|\\
	&\leq  |v^l(x)-v_{n(k)}^{\veps_l}(x)|+\delta +|v_{n(k)}^{\veps_q}(x)-v^q(x)|.
\end{align*}
Sending $k\to\infty$ we get from \eq{key1}${}_2$ that $|v^l(x)-v^q(x)|\leq \delta$. As $x\in I$ is
arbitrary, we conclude that $\|v^l-v^q\|\leq \delta$ whenever $l,q\geq l(\delta)$, establishing the claim.

By completeness of $(\mathcal B(I),\|\cdot\|)$ we thus obtain the existence of a function 
$u\in\mathcal B(I)$ such that $v^l\to u$ uniformly on $I$. 

We claim that $u_{ n(k)}(x)\to u(x)$ for each $x\in I$. To verify this, fix any $x\in I$ and 
any $\delta>0$. Then choose $l$ so large that $\veps_l\leq\frac{\delta}{3}$ and
$\|u-v^l\|\leq\frac{\delta}{3}$. Finally choose $k\geq l$ so large that $|v^l(x)-v_{ n(k)}^{\veps_l}(x)|<\frac{\delta}{3}$
(which is possible according to \eq{key1}${}_2$). Using this together with \eq{key1}${}_1$ we obtain
\begin{align*}
	|u(x)-u_{n(k)}(x)|
	&\leq |u(x)-v^l(x)|+|v^l(x)-v_{n(k)}^{\veps_l}(x)|+|v_{n(k)}^{\veps_l}(x)-u_{n(k)}(x)|\\
	&\leq \|u-v^l\|+|v^l(x)-v_{n(k)}^{\veps_l}(x)|+\veps_l<\delta,
\end{align*}
establishing the claim. Finally, since $(u_{ n(k)})_k$ has uniformly bounded $\veps$-variations, 
we get from the property (R4) above that $u\in\mathcal R(I)$.
\qed

\section{Fra{\v{n}}kov{\'a}'s strategy applied in the multi-variable case}
\label{multi_var_case}
\subsection{Preliminaries}\label{prelims}
For the following background material we refer to \cite{afp,le}.
We fix an open and bounded subset $\Omega\subset\RR^N$, $N\geq 1$.
For convenience, when $N=1$, we assume $\Omega$ is an interval.
In stating compactness results it will be further assumed that $\Omega$ is 
a bounded $BV$ extension domain (cf.\ Definition 3.20 in \cite{afp}).
To simplify the notation we restrict attention to scalar-valued functions $u:\Omega\to \RR$; 
all results generalize routinely to the vector valued case $u:\Omega\to \RR^m$.

A function $u\in L^1(\Omega)$ is of {\em bounded variation} provided
\beq\label{Var}
	\V u:=\sup\Big\{\int_\Omega u\dv \vp\, dx\,|\, 
	\vp\in[C_c^1(\Omega)]^{N}, \|\vp\|_\infty\leq 1 \Big\}<\infty.
\eeq
We shall make repeated use of the fact that $\V$ is lower semi-continuous with respect to
$L^1$-convergence: if $u_n\to u$ in $L^1(\Omega)$, then $\V u\leq\liminf_n \V u_n$; 
cf.\ Remark 3.5 in \cite{afp}.
\begin{remark}\label{notns_of_varn}
	For the one-variable case $N=1$ with $\Omega=(a,b)$, the variation $\V$ relates to the pointwise variation 
	$\var$ in \eq{var} as follows. Defining the {\em essential variation} of $u\in L^1(a,b)$ by
	\[\essen\var u:=\inf\{\var w\,|\, \text{$w$ is a version of $u$ on $(a,b)$}\},\]
	we have (see Section 3.2 in \cite{afp}): every $u\in BV(a,b)$ has 
	a version $\bar u\in\bpv(a,b)$ satisfying
	\beq\label{essvar}
		\var \bar u=\essen\var u=\V u.
	\eeq
\end{remark}
The set $\bv(\Omega)$ of functions of bounded variation is a Banach space when equipped 
with the norm
\[\|u\|_{\bv}:=\|u\|_1+\V u.\]
Furthermore, $\bv(\Omega)$ enjoys the following compactness criterion (cf.\ Theorem 3.23 in \cite{afp}):
\begin{theorem}\label{multi_var_compact}
	Let $\Omega$ be an open and bounded $BV$ extension domain in $\RR^N$ ($N\geq1$), 
	and assume $(u_n)$ is bounded in $(\bv(\Omega),\|\cdot\|_{\bv})$. Then there is a subsequence 
	$(u_{n(k)})\subset(u_n)$ and a function $u\in\bv(\Omega)$ such that $u_{n(k)}\to u$ in $L^1(\Omega)$.
\end{theorem}
We shall provide extensions of this criterion by applying Fra{\v{n}}kov{\'a}'s strategy for
the proof of Theorem \ref{fr_thm}. The first step is to generalize the notion of
$\veps$-variation to functions of several variables.

\subsection{$(\veps,p)$-variation and $p$-regulated functions}\label{eps_p_varns_p_reg_fncs}
Our main objective is to provide extensions of Theorem \ref{multi_var_compact}
which guarantees $L^1$-convergence of a subsequence.
It is therefore natural to seek a notion of $\veps$-variation for general $L^1$-function.
At the same time we want to add flexibility in how the distance between 
a given function and its $\bv$-approximants is measured. In extending 
Fra{\v{n}}kov{\'a}'s setup to multi-variable functions we therefore replace the uniform norm by
any $L^p$-norm, and generalize Definition \ref{epsilon_varn} as follows:
\begin{definition}\label{eps_p_varn_defn}
	Let $u\in L^1(\Omega)$. For $p\in[1,\infty]$  and $\veps>0$  we define 
	the {\em $(\veps,p)$-variation} of $u$ by
	\beq\label{eps_p_varn}
		\epsp  u:=\inf_{v\in \mathcal V_p(u;\veps)} \V v,
	\eeq
	where 
	\beq\label{multi_var_V}
		\mathcal V_p(u;\veps):=\{v\in\bv(\Omega)\,|\, \|u-v\|_p\leq \veps\},
	\eeq
	with the convention that infimum over the empty set is $\infty$. 
\end{definition}

\begin{remark}
	It might seem more natural to restrict the definition of $\epsp u$ to functions $u\in L^p(\Omega)$.
	However, our primary goal is to provide extensions of Theorem \ref{multi_var_compact},
	and this works out with Definition \ref{eps_p_varn_defn} as stated.
	
	We note that already in the case $N=1$, $\Omega=(a,b)$, and $p=\infty$, i.e., 
	the setting closer to that of Fra{\v{n}}kov{\'a} \cite{fr}, our setup is different 
	from hers. Specifically, we consider $L^1$-functions which are 
	really equivalence classes of functions agreeing almost everywhere, the distance
	between a function and its $BV$-approximants is measured in $L^\infty$, and 
	everywhere pointwise convergence plays no role in the present analysis.
%
\end{remark}

We record some immediate consequences of Definition \ref{eps_p_varn_defn}.
\begin{lemma}\label{monotonicity}
	Let $u\in L^1(\Omega)$. Then, for all $p\in[1,\infty]$ and $\veps>0$ we have:
	\begin{enumerate}
		\item If $\bar u$ is a version of $u$, then
		\[\epsp u=\epsp \bar u.\]
		\item If $u\in\bv(\Omega)$, then 
		\[\epsp  u\leq \V u<\infty.\]
		\item If $0<\veps_0<\veps$, then
		\[\epsnaughtp u\geq \epsp u.\]
\end{enumerate}

\end{lemma}
Next, motivated by the characterization (R3) of regulated functions of one variable, we 
introduce the class of {\em $p$-regulated} functions:
\begin{definition}\label{p_reg}
	For $u\in L^1(\Omega)$ and $p\in[1,\infty]$, we say that $u$ is {\em $p$-regulated}
	provided $\epsp u<\infty$ for every $\veps>0$. We set
	\[\mathcal R_p(\Omega):=\{u\in L^1(\Omega)\,|\, \text{$u$ is $p$-regulated}\,\}.\]
\end{definition}
Since $\Omega$ is assumed bounded,
we have $\mathcal R_p(\Omega)\subset\mathcal R_q(\Omega)$ whenever 
$1\leq q< p\leq \infty$. (This follows since $\|f\|_q\leq C\|f\|_p$, where $C=C(p,q,\Omega)$.) 
Furthermore, we have:
\begin{lemma}\label{basic_props}
        For any $p\in[1,\infty]$, $\mathcal R_p(\Omega)$ is a subspace of $L^1(\Omega)$ 
        which is closed under $L^p$-convergence. For $p\in[1,\infty)$ we have 
        $L^p(\Omega)\subset\mathcal R_p(\Omega)$; in particular, 
        $\mathcal R_1(\Omega)\equiv L^1(\Omega)$.
\end{lemma}
\begin{proof}
	Using Definition \ref{p_reg} it is routine to verify that $\mathcal R_p(\Omega)$ is closed 
	under addition and scalar multiplication. Next, for $p\in[1,\infty]$, 
	assume $(u_n)\subset\mathcal R_p(\Omega)$ 
	and $\|u_n- u\|_p\to0$. Since $\Omega$ is bounded we then have $\|u_n-u\|_1\to0$; and
	since $(u_n)\subset L^1(\Omega)$ by assumption, we get that $u\in L^1(\Omega)$.
	Fix any $\veps>0$. Choose $n$ so that $\|u-u_n\|_p\leq\frac{\veps}{2}$.
	Since $u_n\in\mathcal R_p(\Omega)$ there is a $v\in\bv(\Omega)$ with 
	$\|u_n-v\|_p\leq\frac{\veps}{2}$, so that $\|u-v\|_p\leq\veps$.
	As $\veps>0$ is arbitrary, this shows that $u\in\mathcal R_p(\Omega)$, establishing 
	closure of $\mathcal R_p(\Omega)$ under $L^p$-convergence.
	
	Next, if $p\in[1,\infty)$, then $C_c^\infty(\Omega)$ is dense in $L^p(\Omega)$.
	Therefore, given $u\in L^p(\Omega)$ and $\veps>0$ there is a 
	$v\in C_c^\infty(\Omega)\subset \bv(\Omega)$ with $\|u-v\|_p\leq\veps$. 
	Thus, $u\in L^1(\Omega)$ and $\epsp u<\infty$ for each $\veps>0$, so that 
	$u\in\mathcal R_p(\Omega)$. This establishes $L^p(\Omega)\subset \mathcal R_p(\Omega)$
	for $p\in[1,\infty)$.  
	Finally, since $\mathcal R_1(\Omega)\subset L^1(\Omega)$ by definition, this gives 
	$\mathcal R_1(\Omega)\equiv L^1(\Omega)$.
\end{proof}
It follows from part (2) of Lemma \ref{monotonicity} that 
$\bv(\Omega)\subset\mathcal R_p(\Omega)$ for all $p\in[1,\infty]$. 
The following example shows that this inclusion is strict in all cases.
\begin{example}\label{ex}
	Let $\Omega\subset\RR^N$ ($N\geq1$) be open and bounded. By translation it is not 
	restrictive to assume that $\Omega$ contains the origin. Fix $r_0>0$ such that 
	$B_{r_0}(0)\subset\subset \Omega$, and a strictly decreasing sequence of radii 
	$(r_k)_{k\geq 1}$ with $r_1<r_0$ and $r_k\downarrow 0$. Also, fix a strictly 
	decreasing sequence $(\beta)_{k\geq0}$ of positive numbers with $\beta_k\downarrow 0$, 
	and define the radial function $u:\Omega\to\RR$ by
	\beq\label{u}
		u(x):=\sum_{k=0}^\infty \beta_k \chi_{I_k}(|x|),
	\eeq
	where $I_k:=[r_{2k+1},r_{2k}]$ and $\chi_A$ denotes the indicator function of a set $A$.
	Then $u\in L^\infty(\Omega)\subset L^1(\Omega)$ and we have 
	(with $\omega_N=$area of unit ball in $\RR^N$ for $N\geq2$, and $\omega_1=1$)
	\[\V u=\omega_N\sum_{k=0}^\infty\beta_k(r_{2k}^{N-1}+r_{2k+1}^{N-1}).\]
	Choosing, for $k\geq 0$,
	\beq\label{r_beta_choices}
		r_k=(\textstyle\frac{1}{k+1})^\frac{1}{2N}r_0 \qquad\text{and}
		\qquad \beta_k=(\textstyle\frac{1}{k+1})^\frac{1}{2},
	\eeq
	gives
	\[\V u\gtrsim\sum_{k=0}^\infty\beta_k r_{2k+1}^{N-1}
	\gtrsim\sum_{k=0}^\infty\textstyle\frac{1}{k+1}=\infty,\]
	so that $u\notin \bv(\Omega)$. To show that $u\in\mathcal R_p(\Omega)$ 
	for all $p\in[1,\infty]$, it suffices to verify that $u\in\mathcal R_\infty(\Omega)$ 
	(see comment before Lemma \ref{basic_props}). For this define
	\beq\label{u_n}
		u_n(x):=\sum_{k=0}^n \beta_k \chi_{I_k}(r),
	\eeq
	and let, for $\veps>0$, $n_\veps$ be the smallest integer such that $\beta_{n_\veps}\leq\veps$.
	Then
	\[u_{n_\veps}\in\bv(\Omega)\qquad\text{and}\qquad \|u-u_{n_\veps}\|_\infty\leq\veps,\]
	showing that $\epsinf u<\infty$. As $\veps>0$ is arbitrary, we have $u\in\mathcal R_\infty(\Omega)$.
\end{example}

Next we generalize Definition \ref{unif_eps_varns} to the setting of $(\veps,p)$-variation.
\begin{definition}\label{multi_var_unif_eps_varns}
	Let $p\in[1,\infty]$. A set of functions $\mathcal F\subset \mathcal R_p(\Omega)$ has {\em uniformly
	bounded $(\veps,p)$-variations} provided
	\[\sup_{u\in\mathcal F}\,\,\epsp u<\infty\qquad\text{for every $\veps>0$}.\]
\end{definition}
The following example builds on Example \ref{ex} and shows that for any $p\in[1,\infty]$
there are sequences with uniformly bounded $(\veps,p)$-variations that are unbounded 
in $\bv(\Omega)$.
\begin{example}\label{2nd_ex}
	Fix $p\in[1,\infty]$ and consider the sequence $(u_n)$ defined in \eq{u_n}.
	Note that $u_n\in\bv(\Omega)$ for all $n$ and that $\V u_n$ increases monotonically 
	without bound as $n\uparrow\infty$ (cf.\ Example \ref{ex}). We have
	\[\|u-u_n\|_p\lesssim\|u-u_n\|_\infty\to 0,\]
	and $\|u-u_n\|_p$ decreases monotonically to $0$ as $n\uparrow\infty$.
	For given $\veps>0$ let $m_\veps$ be the smallest integer $m$ such that 
	$\|u-u_m\|_p\leq\veps$. For $n\leq m_\veps$ we have $\V u_n\leq\V u_{m_\veps}<\infty$,
	so that 
	\[\epsp u_n\leq \V u_n\leq\V u_{m_\veps},\qquad n\leq m_\veps. \]
	On the other hand, for $n> m_\veps$ we have
	\[\|u_n-u_{m_\veps}\|_p<\|u-u_{m_\veps}\|_p\leq\veps,\]
	showing that $\epsp u_n\leq\V u_{m_\veps}$ holds also in this case.
	It follows that 
	\[\sup_n\,\,\epsp u_n\leq \V u_{m_\veps}<\infty. \]
	As $\veps>0$ is arbitrary, this shows that $(u_n)$ has uniformly bounded $(\veps,p)$-variations
	for any $p\in[1,\infty]$.
\end{example}

\subsection{Multi-variable Fra{\v{n}}kov{\'a} theorem}\label{gen'zd_Frankova}
In this subsection we formulate and prove our main result on extensions of the 
$\bv$-compactness criterion in Theorem \ref{multi_var_compact}. 
We first record the following observation, 
cf.\ Observation \ref{obs1} in Section \ref{frankova_helly_extn}.
\begin{observation}\label{obs2}
        Let $\Omega$ be an open and bounded $\bv$ extension domain in $\RR^N$ ($N\geq1$)
        and $p\in[1,\infty]$. Assume $(z_n)\subset\mathcal R_p(\Omega)$ is bounded in $L^1(\Omega)$ 
        and has uniformly bounded $(\veps,p)$-variations.
        Then, for each $\veps>0$ there is a finite number $K_\veps$ and a 
        sequence $(z_n^\veps)\subset\bv(\Omega)$ satisfying
        \[\V z_n^\veps\leq K_\veps,\qquad \|z_n-z_n^\veps\|_p\leq\veps \qquad\text{for all $n\geq1$}.\]
        Since $\Omega$ is assumed bounded, there is a constant $C=C(\Omega,p)$ so that
        \[\|z_n^\veps\|_1\leq \|z_n^\veps-z_n\|_1+\|z_n\|_1\leq C\|z_n^\veps-z_n\|_p+\|z_n\|_1\leq C\veps+\|z_n\|_1.\]
        It follows that, for each fixed $\veps$, the sequence $(z_n^\veps)$ 
        meets the assumptions in Theorem \ref{multi_var_compact}. Thus, for each $\veps>0$, there is 
        a subsequence $(z_{n(k)}^\veps)\subset(z_n^\veps)$ and a $z^\veps\in\bv(\Omega)$ with
        $z_{n(k)}^\veps\to z^\veps$ in $L^1(\Omega)$.
\end{observation}
We now have:
\begin{theorem}\label{multi_var_frankova_any_p}
	Let $\Omega$ be an open and bounded $\bv$ extension domain in $\RR^N$ ($N\geq1$), and let 
	$p\in[1,\infty]$. Assume the sequence $(u_n)$ 
	is bounded in $L^1(\Omega)$ and has uniformly bounded $(\veps,p)$-variations.
	Then there is a subsequence $(u_{n(k)})\subset (u_n)$ and a function $u\in L^1(\Omega)$ 
	such that $u_{n(k)}\to u$ in $L^1(\Omega)$.
\end{theorem}
\begin{remark}
	Note that it is not claimed that the limit function $u$ belongs to 
	$\mathcal R_p(\Omega)$. However, this does hold when $p=1$ or $p=\infty$. 
	For $p=1$ this is immediate since $u\in L^1(\Omega)\equiv \mathcal R_1(\Omega)$ 
	(Lemma \ref{basic_props}). Furthermore, lower semi-continuity of $\epsone\!$ with respect to 
	$L^1$-convergence (Proposition \ref{lsc_p=1} below) yields 
	\[\epsone u\leq\liminf_k \,\,\epsone u_{n(k)}\qquad\text{for each $\veps>0$.}\] 
	When $p=\infty$ we again have lower semicontinuity of $\epsinf\!$ with respect to 
	$L^1$-convergence (Proposition \ref{lsc_p=infty} below). Thus, also for the case $p=\infty$ we have
	\[\epsinf u\leq\liminf_k \,\,\epsinf u_{n(k)}\qquad\text{for each $\veps>0$.}\] 
	As the sequence $(u_n)$ is assumed to have uniformly bounded $(\veps,\infty)$-variations, it follows that 
	$\epsinf u<\infty$ for each $\veps>0$, i.e., $u\in\mathcal R_\infty(\Omega)$.  
\end{remark}
\medskip\noindent
{\it Proof of Theorem \ref{multi_var_frankova_any_p}.} 
Fix a strictly decreasing sequence $(\veps_l)$ converging to zero. 
For $l=1$ apply Observation \ref{obs2} to the original sequence $(u_n)$ with $\veps=\veps_1$ 
to get a sequence-subsequence pair $(v_n^{\veps_1})\supset(v_{n_1(k)}^{\veps_1})$ 
in $\bv(\Omega)$ and a $v^1\in\bv(\Omega)$ satisfying 
\[\|u_{n_1(k)}-v_{n_1(k)}^{\veps_1}\|_p\leq \veps_1\quad \text{for all $k\geq 1$, and}\qquad
v_{n_1(k)}^{\veps_1}\overset{k}{\to} v^1\quad \text{in $L^1(\Omega)$.}\]
For $l=2$ apply Observation \ref{obs2} to the sequence $(u_{n_1(k)})$ with $\veps=\veps_2$ 
to get a sequence-subsequence pair $(v_{n_1(k)}^{\veps_2})\supset(v_{n_2(k)}^{\veps_2})$ 
in $\bv(\Omega)$ and a $v^2\in\bv(\Omega)$ satisfying 
\[\|u_{n_2(k)}-v_{n_2(k)}^{\veps_2}\|_p\leq \veps_2\quad \text{for all $k\geq 1$, and}\qquad
v_{n_2(k)}^{\veps_2}\overset{k}{\to}v^2\quad \text{in $L^1(\Omega)$.}\]
Continuing in this manner we obtain for each index $l$ a sequence-subsequence pair
$(v_{n_{l-1}(k)}^{\veps_l})\supset(v_{n_l(k)}^{\veps_l})$ in $\bv(\Omega)$ and a $v^l\in\bv(\Omega)$ 
satisfying
\[\|u_{n_l(k)}-v_{n_l(k)}^{\veps_l}\|_p\leq\veps_l\quad \text{for all $k\geq 1$, and}\qquad
v_{n_l(k)}^{\veps_l}\overset{k}{\to}v^l\quad \text{in $L^1(\Omega)$.}\]
Next, for $l\geq 1$ fixed, consider the diagonal index sequence $(n_k(k))_{k\geq l}$,
which is a subsequence of $(n_l(j))_{j\geq 1}$. With $ n(k):=n_k(k)$ we therefore get
that: for each $l\geq 1$, there holds
\beq\label{key2}
	\|u_{ n(k)}-v_{ n(k)}^{\veps_l}\|_p\leq\veps_l\quad \text{for all $k\geq l$, and}\quad 
	v_{ n(k)}^{\veps_l}\overset{k}{\to} v^l\quad \text{in $L^1(\Omega)$.}
\eeq
We claim that $(v^l)$ is a Cauchy sequence in $L^1(\Omega)$. To show this let 
$C=C(\Omega,p)$ be a constant such that $\|f\|_1\leq C\|f\|_p$ for all $f\in L^p(\Omega)$.
Then, for any $\delta>0$ 
we first choose an index $l(\delta)$ so that $\veps_l\leq\frac{\delta}{2C}$ for $l\geq l(\delta)$. 
Then, for $l,q\geq l(\delta)$ and  any $k\geq \max(l,q)$, \eq{key2}${}_1$ gives that
\begin{align*}
	\|v^l-v^q\|_1
	&\leq \|v^l-v_{ n(k)}^{\veps_l}\|_1+\|v_{n(k)}^{\veps_l}-u_{n(k)}\|_1
	+\|u_{n(k)}-v_{n(k)}^{\veps_q}\|_1+\|v_{n(k)}^{\veps_q}-v^q\|_1\\
	&\leq \|v^l-v_{ n(k)}^{\veps_l}\|_1+C\veps_l+C\veps_q+\|v_{n(k)}^{\veps_q}-v^q\|_1\\
	&\leq \|v^l-v_{ n(k)}^{\veps_l}\|_1+\delta +\|v_{n(k)}^{\veps_q}-v^q\|_1.
\end{align*}
Sending $k\to\infty$ we get from \eq{key2}${}_2$ that $\|v^l-v^q\|_1\leq \delta$ 
whenever $l,q\geq l(\delta)$, establishing the claim.

By completeness of $L^1(\Omega)$ we thus obtain the existence of a function 
$u\in L^1(\Omega)$ such that $v^l\to u$ in $L^1(\Omega)$. 
We claim that $u_{ n(k)}\to u$ in $L^1(\Omega)$. To verify this, fix any $\delta>0$ and choose $l$
so large that $\veps_l\leq\frac{\delta}{3C}$ (with $C$ as above) and $\|u-v^l\|_1\leq\frac{\delta}{3}$. According to 
\eq{key2}${}_1$ we therefore have, for any $k\geq l$, that
\begin{align*}
	\|u-u_{n(k)}\|_1
	&\leq \|u-v^l\|_1+\|v^l-v_{n(k)}^{\veps_l}\|_1+\|v_{n(k)}^{\veps_l}-u_{n(k)}\|_1\\
	&\leq \textstyle\frac{\delta}{3}+\|v^l-v_{n(k)}^{\veps_l}\|_1+C\|v_{n(k)}^{\veps_l}-u_{n(k)}\|_p\\
	&
	\leq\textstyle\frac{2\delta}{3}+\|v^l-v_{n(k)}^{\veps_l}\|_1.
\end{align*}
Finally choose $k\geq l$ so large that $\|v^l-v_{n(k)}^{\veps_l}\|_1<\frac{\delta}{3}$
(possible according to \eq{key2}${}_2$), giving $\|u-u_{n(k)}\|_1\leq\delta$. 
This establishes the claim, and concludes the proof of Theorem \ref{multi_var_frankova_any_p}. 
\qed
\medskip

In the following two subsections we provide additional results for $p=1$ and $p=\infty$.

\subsection{The case $p=1$}\label{multi_var_case_p=1}
We shall establish attainment of $(\veps,1)$-variation, right-continuity with respect to 
$\veps$, and lower semi-continuity with respect to $L^1$-convergence. The arguments 
make use of Theorem \ref{multi_var_compact} and thus require that $\Omega$ is
a $\bv$ extension domain. 
\begin{proposition}\label{attained_1}
	Let $\Omega\subset\RR^N$ ($N\geq1$) be an open and bounded $\bv$ extension domain 
	and assume $u\in L^1(\Omega)$ and $\veps>0$. Then there is a function
	$v\in \bv(\Omega)$ with $\|u-v\|_1\leq\veps$ and $\V v=\epsone u$.
\end{proposition}
\begin{proof}
	Recall from Lemma \ref{basic_props} that $\mathcal R_1(\Omega)=L^1(\Omega)$, 
	so that $\epsone u<\infty$. According to Definition \ref{eps_p_varn_defn} 
	there is a sequence $(v_k)\subset \bv(\Omega)$ with $\|v_k-u\|_1\leq\veps$ 
	for all $k\geq 1$ and such that
	\beq\label{v_k_2}
		\epsone u=\lim_k \V v_k.
	\eeq
	Also, for all $k\geq1$,
	\[\|v_k\|_1\leq \|v_k-u\|_1+\|u\|_1\leq\veps+\|u\|_1.\]
	Thus $(v_k)$ is bounded in $\bv(\Omega)$, and Theorem \ref{multi_var_compact} 
	gives a subsequence $(v_{k(j)})\subset (v_k)$ and a $v\in\bv(\Omega)$ so that 
	\beq\label{v_k_3}
		v_{k(j)}\to v\quad\text{in $L^1(\Omega)$.}
	\eeq
	According to lower semi-continuity of $\V$ with respect to $L^1$-convergence 
	we have $\V v\leq\liminf_j \V v_{k(j)}$.
	Together with \eq{v_k_2} this gives
	\beq\label{v_k_4}
		\V v\leq\liminf_j \V v_{k(j)}=\lim_j \V v_{k(j)}=\epsone u.
	\eeq
	On the other hand, we have for any $j$ that
	\[\|u-v\|_1\leq \|u-v_{k(j)}\|_1+\|v_{k(j)}-v\|_1\leq\veps+\|v_{k(j)}-v\|_1.\]
	By sending $j\to\infty$ and using \eq{v_k_3} we therefore get $\|u-v\|_1\leq\veps$.
	Definition \ref{eps_p_varn_defn} therefore gives $\epsone u\leq \V v,$ showing that
	$\epsone u= \V v$.
\end{proof}
We next apply this last result to establish right-continuity with respect to $\veps$.
\begin{proposition}\label{right_cont}
	Let $\Omega\subset\RR^N$ ($N\geq1$) be an open and bounded $\bv$ extension 
	domain and assume $u\in L^1(\Omega)$ and $\veps_0>0$. Then
	\beq\label{r_cont}
		\lim_{\veps\downarrow\veps_0} \,\,\epsone u=\epsnaughtone u.
	\eeq
\end{proposition}
\begin{proof}
	Since $L^1(\Omega)=\mathcal R_1(\Omega)$ the 
	right-hand side of \eq{r_cont} is finite. Also, according to part (3) of
	Lemma \ref{monotonicity}, we have $\epsone u\leq \epsnaughtone u$
	whenever $\veps>\veps_0$. Denoting the limit on the left-hand side of 
	\eq{r_cont} by $L_0$, it follows that $L_0$ exists as a finite number 
	satisfying
	\beq\label{one_way}
		L_0\leq \epsnaughtone u.
	\eeq
	For the opposite inequality we use Proposition \ref{attained_1} to select, 
	for each $\veps>\veps_0$, a function $v^\veps\in\bv(\Omega)$ with
	\[\|u-v^\veps\|_1\leq\veps\qquad\text{and}\qquad \V v^\veps=\epsone u\leq \epsnaughtone u<\infty.\]
	Since,
	\[\|v^\veps\|_1\leq\|v^\veps-u\|_1+\|u\|_1\leq \veps+\|u\|_1<\infty,\]
	we have that $\{v^\veps\,|\, \veps>\veps_0\}$ is bounded in $\bv(\Omega)$.
	It follows from Theorem \ref{multi_var_compact} that there is a sequence $(\veps_n)_{n\geq1}$
	with $\veps_n\downarrow \veps_0$ and a function $v\in\bv(\Omega)$ so that $v_n:=v^{\veps_n}$
	satisfies $v_n\to v$ in $L^1(\Omega)$. This gives
	\beq\label{one_way_1}
		L_0=\lim_{\veps\downarrow\veps_0} \,\,\epsone u=\lim_{n} \V v_n\equiv\liminf_n\V v_n\geq \V v.
	\eeq
	On the other hand, 
	\[\|u-v\|_1\leq\|u-v_n\|_1+\|v_n-v\|_1\leq \veps_n+\|v_n-v\|_1,\]
	so that sending $n\to\infty$ yields $\|u-v\|_1\leq\veps_0$. It follows 
	from Definition \ref{eps_p_varn_defn} that 
	\[\epsnaughtone u\leq\V v,\]
	which together with \eq{one_way_1} gives $\epsnaughtone u\leq L_0.$
\end{proof}
We next show how the two previous propositions yield lower semi-continuity 
of $\epsone u$ with respect to $L^1$-convergence. 
\begin{proposition}\label{lsc_p=1}
	Let $\Omega\subset\RR^N$ ($N\geq1$) be an open and bounded $\bv$ extension 
	domain. If $u_n\to u$ in $L^1(\Omega)$, then 
	\[\epsone u\leq\liminf_n\,\,\epsone u_n \qquad\text{for each $\veps>0$.}\]
\end{proposition}
\begin{proof}
	Fix $\veps>0$. According to Proposition \ref{attained_1}, for each $n\geq 1$ there is a
	function $v_n\in\bv(\Omega)$ with
	\[\|u_n-v_n\|_1\leq\veps\qquad\text{and}\qquad \epsone u_n=\V v_n.\]
	Let $\delta>0$, and choose $M=M(\delta)\in\NN$ so that $\|u_n-u\|_1\leq\delta$ 
	for $n\geq M$. For such $n$ we have
	\[\|v_n-u\|_1\leq \|v_n-u_n\|_1+\|u_n-u\|_1\leq \veps+\delta,\]	
	and therefore
	\[\epsdeltaone u\leq \V v_n=\epsone u_n\qquad\text{for each $n\geq M$.}\]
	It follows that 
	\[\epsdeltaone u\leq\liminf_n\,\,\epsone u_n\qquad\text{whenever $\delta>0$.}\]
	Finally, sending $\delta\downarrow 0$ and using Proposition \ref{right_cont} give
	\[\epsone u=\lim_{\delta\downarrow 0}\,\,\epsdeltaone u\leq \liminf_n\,\,\epsone u_n,\]
	establishing the claim.
\end{proof}

\subsection{The case $p=\infty$}\label{multi_var_case_p=infty}
We shall establish the results corresponding to the previous three propositions 
also for $p=\infty$. The proofs for the two first are similar to those for the case $p=1$, 
while for the proof of lower semi-continuity we follow Fra{\v{n}}kov{\'a}'s argument for 
(R4) in Section \ref{frankova_helly_extn} (cf.\ Proposition 3.6 in \cite{fr}).
In each case there is the added ingredient that $L^1$-convergence implies almost everywhere 
convergence along a subsequence. We first verify that the infimum in the definition of 
$(\veps,\infty)$-variation is attained.

\begin{proposition}\label{attained_infty}
	Let $\Omega\subset\RR^N$ ($N\geq1$) be an open and bounded $\bv$ extension domain,
	and assume $u\in \mathcal R_\infty(\Omega)$ and $\veps>0$. Then there is a function
	$v\in \bv(\Omega)$ with $\|u-v\|_\infty\leq\veps$ and $\V v=\epsinf u$.
\end{proposition}
\begin{proof} Fix $u$ and $\veps$ as stated.
	According to Definition \ref{eps_p_varn_defn} 
	there is a sequence $(v_k)\subset \bv(\Omega)$ with $\|v_k-u\|_\infty\leq\veps$ 
	for all $k\geq 1$ and such that
	\beq\label{v_k_5}
		\lim_k \V v_k=\epsinf u<\infty.
	\eeq
	Also, 
	\[\|v_k\|_1\leq \|v_k-u\|_1+\|u\|_1\leq|\Omega|  \|v_k-u\|_\infty+\|u\|_1
	\leq |\Omega|\veps+\|u\|_1<\infty.\]
	It follows that $(v_k)$ is bounded in $\bv(\Omega)$,	 
	 so that Theorem \ref{multi_var_compact} gives a subsequence $(v_{k(j)})\subset (v_k)$
	and a $v\in\bv(\Omega)$ with 
	\beq\label{v_k_6}
		v_{k(j)}\to v\quad\text{in $L^1(\Omega)$, and}\quad \V v\leq\liminf_j \V v_{k(j)}.
	\eeq
	Note that \eq{v_k_6}${}_2$ and \eq{v_k_5} give
	\beq\label{v_k_7}
		\V v\leq\liminf_j \V v_{k(j)}\equiv\lim_j \V v_{k(j)}=\epsinf u.
	\eeq
	For the opposite inequality we use that the $L^1$-convergence in \eq{v_k_6}${}_1$ 
	implies almost everywhere convergence along a further subsequence $v_{k(j(i))}=:w_i$:
	\beq\label{w_i_to_v}
		w_i(x)\to v(x)\qquad\text{for all $x\in\Omega_0$,}
	\eeq
	where $\Omega_0\subset\Omega$ has full measure.
	Next, as $\|u-w_i\|_\infty\leq \veps$ for each $i\geq 1$, we have that for each $i\geq 1$
	there is a set $\Omega_i\subset\Omega$ of full measure with
	\beq\label{u-w_i}
		|u(x)-w_i(x)|\leq \veps\qquad\text{for each $x\in\Omega_i$.}
	\eeq
	With 
	\[\Omega':={\textstyle\bigcap}_{i=0}^\infty \Omega_i,\]
	we have that $\Omega'$ has full measure, while \eq{u-w_i} gives
	\beq\label{u-v}
		|u(x)-v(x)|\leq |u(x)-w_i(x)|+|w_i(x)-v(x)|\leq\veps +|w_i(x)-v(x)|
	\eeq
	for each $x\in\Omega'$ and all $i\geq 1$.
	Sending $i\to\infty$ we obtain from \eq{w_i_to_v} that
	\[|u(x)-v(x)|\leq \veps\qquad\text{for each $x\in\Omega'$,}\]
	so that 
	\[\|u-v\|_\infty\leq \veps.\]
	As $v\in\bv(\Omega)$, Definition \ref{eps_p_varn_defn} gives $\epsinf u\leq \V v$. 
	Together with \eq{v_k_7} this shows that $\epsinf u= \V v$.
\end{proof}
We next establish right-continuity of $\epsinf u$ with respect to $\veps$.
\begin{proposition}\label{right_cont_inf}
	Let $\Omega\subset\RR^N$ ($N\geq1$) be an open and bounded $\bv$ extension domain,
	and let $u\in \mathcal R_\infty(\Omega)$ and $\veps_0>0$. Then
	\beq\label{r_cont_inf}
		\lim_{\veps\downarrow\veps_0} \,\,\epsinf u=\epsnaughtinf u.
	\eeq
\end{proposition}
\begin{proof}
	Since $u\in \mathcal R_\infty(\Omega)$ the right-hand side of \eq{r_cont_inf} is finite.
	According to Lemma \ref{monotonicity}, we have $\epsinf u\leq \epsnaughtinf u$
	for $\veps>\veps_0$. Denoting the limit on the left-hand side of 
	\eq{r_cont_inf} by $L_0$, it follows that $L_0$ exists as a finite number 
	satisfying
	\beq\label{one_way_inf_1}
		L_0\leq \epsnaughtinf u.
	\eeq
	For the opposite inequality we use Proposition \ref{attained_infty} to select, 
	for each $\veps>\veps_0$, a function $v^\veps\in\bv(\Omega)$ with
	\[\|u-v^\veps\|_\infty\leq\veps\qquad\text{and}\qquad \V v^\veps=\epsinf u\leq L_0.\]
	Also,
	\[\|v^\veps\|_1\leq\|v^\veps-u\|_1+\|u\|_1\leq |\Omega|\|v^\veps-u\|_\infty+\|u\|_1\leq |\Omega|\veps+\|u\|_1<\infty,\]
	so that $\{v^\veps\,|\, \veps>\veps_0\}$ is bounded in $\bv(\Omega)$.
	According to Theorem \ref{multi_var_compact} there is a sequence $(\veps_n)_{n\geq1}$
	with $\veps_n\downarrow \veps_0$ and a function $v\in\bv(\Omega)$ so that $v_n:=v^{\veps_n}$
	satisfies $v_n\to v$ in $L^1(\Omega)$. 
	This gives
	\beq\label{one_way_inf}
		L_0=\lim_{\veps\downarrow\veps_0} \,\,\epsinf u=\lim_{n} \V v_n\equiv\liminf_n\V v_n\geq \V v.
	\eeq
	Next, since $v_n\to v$ in $L^1(\Omega)$, there is a full-measure set $\Omega_0\subset \Omega$
	such that 
	\beq\label{next}
		v_n(x)\to v(x)\qquad\text{for each $x\in\Omega_0$.}
	\eeq
	Also, since $\|u-v_n\|_\infty\leq\veps_n$ for each $n\geq 1$, there are full measure sets $\Omega_n$
	such that
	\[|u(x)-v_n(x)|\leq \veps_n\qquad\text{for each $x\in\Omega_n$, $n\geq 1$.}\]
	Therefore, with 
	\[\Omega':={\textstyle\bigcap}_{n=0}^\infty \Omega_n,\]
	we have that $\Omega'$ is of full measure and such that
	\[|u(x)-v(x)|\leq |u(x)-v_n(x)|+|v_n(x)-v(x)|\leq \veps_n+|v_n(x)-v(x)|\]
	for each $x\in\Omega'$ and for all $n\geq 1$.
	Sending $n\to\infty$ and using \eq{next} yields $|u(x)-v(x)|\leq\veps_0$ for each $x\in\Omega'$,
	so that
	\[\|u-v\|_\infty\leq\veps_0.\] 
	As $v\in\bv(\Omega)$, Definition \ref{eps_p_varn_defn} gives
	\[\epsnaughtinf u\leq\V v,\]
	which together with \eq{one_way_inf} gives $\epsnaughtinf u\leq L_0$. Combined with \eq{one_way_inf_1}
	this finishes the proof.
\end{proof}
We finally establish lower semi-continuity of $\epsinf$ with respect to $L^1$-convergence.
\begin{proposition}\label{lsc_p=infty}
	Let $\Omega\subset\RR^N$ ($N\geq1$) be an open and bounded $\bv$ extension domain.
	If $(u_n)\subset\mathcal R_\infty(\Omega)$ and $u_n\to u$ in $L^1(\Omega)$, then 
	\beq\label{lsc_infty}
		\epsinf u\leq\liminf_n\,\,\epsinf u_n \qquad\text{for each $\veps>0$.}
	\eeq
\end{proposition}
\begin{proof}
	Fix $\veps>0$. If the right-hand side of \eq{lsc_infty} is infinite, there is nothing to prove.
	So assume $L:=\liminf_n\,\,\epsinf u_n<\infty$. We select a subsequence $(u_{n(k)})\subset(u_n)$
	with
	\[\lim_k \,\,\epsinf u_{n(k)}=L.\]
	As $u_{n(k)}\to u$ in $L^1(\Omega)$ we can extract a further subsequence 
	$(u_{n(k(j))})\subset (u_{n(k)})$ so that 
	\beq\label{pw1}
		u_{n(k(j))}(x)\to u(x) \qquad\text{for each $x\in \Omega_{-1}$,}
	\eeq
	where $\Omega_{-1}\subset \Omega$ has full measure.
	Since each $u_{n(k(j))}\in\mathcal R_\infty(\Omega)$, Proposition \ref{attained_infty} gives 
	a sequence $(v_j)\subset \bv(\Omega)$ with
	\beq\label{vj}
		\V v_j=\epsinf u_{n(k(j))}\qquad\text{and}\qquad \|v_j-u_{n(k(j))}\|_\infty\leq \veps
		\qquad\text{for each $j\geq1$.}
	\eeq
	We therefore have
	\beq\label{next_next}
		\lim_j \V v_j = \lim_j \epsinf u_{n(k(j))} =\lim_k \,\,\epsinf u_{n(k)}=L<\infty,
	\eeq
	and
	\begin{align*}
		\|v_j\|_1&\leq\|v_j-u_{n(k(j))}\|_1+\|u_{n(k(j))}\|_1\\
		&\leq |\Omega|\|v_j-u_{n(k(j))}\|_\infty+\|u_{n(k(j))}\|_1
		\leq |\Omega|\veps+ \|u_{n(k(j))}\|_1,
	\end{align*}
	which is uniformly bounded with respect to $j$ since $u_n\to u$ in $L^1(\Omega)$.
	The sequence $(v_j)$ is therefore bounded in $\bv(\Omega)$, and Theorem 
	\ref{multi_var_compact} gives a subsequence $(v_{j(i)})\subset(v_j)$ and 
	a $v\in\bv(\Omega)$ such that $v_{j(i)}\to v$ in $L^1(\Omega)$. According to \eq{next_next}
	and lower semi-continuity of $\V$ with respect to $L^1$-convergence, we get
	\beq\label{var_v}
		\V v\leq \liminf_i \V v_{j(i)}\equiv \lim_j \V v_j=L.
	\eeq
	We then extract a further subsequence $(v_{j(i(h))})\subset (v_{j(i)})$ together with a 
	full measure set $\Omega_0\subset\Omega$ so that
	\beq\label{pw2}
		v_{j(i(h))}(x)\to v(x)\qquad\text{for each $x\in\Omega_0$.}
	\eeq
	Finally, according to \eq{vj}${}_2$ there is, for each $j\geq 1$, a full measure set 
	$\Omega_j\subset\Omega$ such that
	\beq\label{pw3}
		|u_{n(k(j))}-v_j(x)|\leq\veps\qquad\text{for each $x\in\Omega_j$, $j\geq1$.}
	\eeq
	Defining the full measure set  
	\[\Omega':={\textstyle\bigcap}_{j=-1}^\infty \Omega_j,\]
	we therefore get from \eq{pw3} that
	\begin{align*}
		|u(x)-v(x)|&\leq |u(x)-u_{n(k(j(i(h))))}(x)|+|u_{n(k(j(i(h))))}(x)-v_{j(i(h))}(x)|+|v_{j(i(h))}(x)-v(x)|\\
		&\leq |u(x)-u_{n(k(j(i(h))))}(x)|+\veps+|v_{j(i(h))}(x)-v(x)|\qquad\text{for each $x\in\Omega'$.}
	\end{align*}
	Sending $h\to\infty$ and using \eq{pw1} and \eq{pw2} gives $|u(x)-v(x)|\leq\veps$ for 
	all $x\in\Omega'$, so that
	\beq\label{inf_eps}
		\|u(x)-v(x)\|_\infty\leq\veps.
	\eeq
	As $v\in\bv(\Omega)$, Definition \ref{eps_p_varn_defn} together with \eq{var_v} and \eq{inf_eps} 
	give
	\[\epsinf u\leq \V v\leq L.\]
	As $\veps>0$ is arbitrary, this gives \eq{lsc_infty}.
\end{proof}

\subsection{Regulated vs.\ $\infty$-regulated functions of a single variable}\label{reg_vs_infty_reg_1_d}
In this section we clarify the relationship between standard regulated functions (cf.\ Definition
\ref{reg}) and $\infty$-regulated functions (cf.\ Definition \ref{p_reg} with $p=\infty$) defined on 
an open interval $I$. The following result shows that the latter coincide with 
``essentially regulated'' functions. (Recall that $\|\cdot\|$ denotes uniform norm;
step functions were defined in Section \ref{frankova_helly_extn}.)
\begin{proposition}\label{1_d_reg_vs_infty_reg}
	Let $I=(a,b)$ be an open and bounded interval, and $u:I\to\RR$.
	Then the following statements are equivalent:
	\begin{itemize}
		\item[(a)] $u$ is $\infty$-regulated, i.e., $\epsinf u<\infty$ 
		for each $\veps>0$.
		\item[(b)] $u$ can be realized as an $L^\infty$-limit of step functions on $I$.
		\item[(c)] $u$ has a regulated version, i.e., there exists  $\bar u\in\mathcal R(I)$ 
		with $\bar u(x)=u(x)$ for a.a.\ $x\in I$.
	\end{itemize}
 \end{proposition}
\begin{proof}
	(a) $\Rightarrow$ (b): Fix $\veps>0$. By Definition \ref{p_reg} there is a 
	function $v^\veps\in\bv(I)$ with $\|u-v^\veps\|_\infty\leq \frac{\veps}{2}$. According to 
	Theorem 7.2 in \cite{le}, $v^\veps$ has a version $\bar v^\veps$ belonging to 
	$\bpv(I)\subset\mathcal R(I)$.
	It follows from property (R2) in Section \ref{frankova_helly_extn} that there is a step 
	function $w^\veps:I\to\RR$ satisfying $\|\bar v^\veps-w^\veps\|\leq\frac{\veps}{2}$.
	Therefore,
	\begin{align*}
		\|u-w^\veps\|_\infty&\leq \|u-v^\veps\|_\infty+\|v^\veps-w^\veps\|_\infty\\
		&=\|u-v^\veps\|_\infty+\|\bar v^\veps-w^\veps\|_\infty\\
		&\leq\|u-v^\veps\|_\infty+\|\bar v^\veps-w^\veps\|
		\leq\textstyle\frac{\veps}{2}+\frac{\veps}{2}=\veps.
	\end{align*}
	As $w^\veps$ is a step function and $\veps$ is arbitrary, statement (b) follows.
	
	(b) $\Rightarrow$ (c): Assume $(u_n)$ is  a sequence of step functions on $I$ 
	with $\|u-u_n\|_\infty\to 0$.
	For each $n\geq 1$ let $\bar u_n$ denote the right-continuous version of $u_n$, which is obtained 
	from $u_n$ by changing at most a finite number of its values. It follows that 
	$\|u-\bar u_n\|_\infty\to 0$. Therefore, there is a measurable set $E\subset I$ with $|E|=|I|$ 
	and such that $\bar u_n\to u$ uniformly on $E$. We claim that the sequence $(\bar u_n)$ is uniformly 
	Cauchy on {\em all} of $I$, i.e., 
	\beq\label{sub_claim}
		\|\bar u_m-\bar u_n\|\to 0 \qquad\text{as $m,n\to\infty$.}
	\eeq
	Assuming the validity of \eq{sub_claim} for now, we have, by completeness of $(\mathcal B(I),\|\cdot\|)$, 
	that there exists $\bar u\in \mathcal B(I)$ such that $\bar u_n\to \bar u$ uniformly on $I$.
	In particular, according to property (R2) in Section \ref{frankova_helly_extn}, $\bar u\in \mathcal R(I)$.
	Since $\bar u_n\to u$ uniformly on $E$, it follows that $\bar u$ and $u$ agree on $E$, establishing (c). 
	
	It remains to verify \eq{sub_claim}. For this, fix $\veps>0$. Since $\bar u_n\to u$ uniformly on $E$ we have:
	there is an $N_\veps\in\NN$ such that 
	\beq\label{on_E}
		\sup_{y\in E} |\bar u_n(y)-\bar u_m(y)|\leq\veps\qquad\text{whenever $m,n\geq N_\veps$.}
	\eeq
	Fix $m,n\geq N_\veps$ and $x\in I$. Since $\bar u_n$ and $\bar u_m$ are right-continuous 
	step functions, there is a $\delta_x>0$ such that both $\bar u_n$ and $\bar u_m$ are 
	constant on $J_x:=[x,x+\delta_x)$. Since $E\subset I$ has full measure, $E$ is dense in $I$,
	and there is therefore a point $y_x\in E\cap J_x$. As $\bar u_n$ and $\bar u_m$ are 
	constant on $J_x$ we get from \eq{on_E} that 
	\[|\bar u_n(x)-\bar u_m(x)|=|\bar u_n(y_x)-\bar u_m(y_x)|\leq \veps.\]
	As $x\in  I$ is arbitrary it follows that
	\[\sup_{x\in I}|\bar u_n(x)-\bar u_m(x)|\leq \veps.\]
	This shows that $\|\bar u_n-\bar u_m\|\leq\veps$ whenever $m,n\geq N_\veps$, 
	establishing \eq{sub_claim}.
	
	(c) $\Rightarrow$ (a): Assume $\bar u\in\mathcal R(I)$ is a version of $u$, and fix $\veps>0$. 
	According to (R3) in Section \ref{frankova_helly_extn}, we have $\evar \bar u<\infty$, i.e.,
	there is a function $v^\veps\in\bpv(I)$ with $\|\bar u-v^\veps\|\leq \veps$.
	According to Theorem 7.2 in \cite{le}, $v^\veps$ belongs to $\bv(I)$ and we have
	\[\|u-v^\veps\|_\infty= \|\bar u-v^\veps\|_\infty\leq \|\bar u-v^\veps\|\leq \veps.\]
	Thus, $\epsinf u<\infty$, and since $\veps>0$ is arbitrary, we have $u\in\mathcal R_\infty(I)$. 
\end{proof}

%
%
%
%
%
%
%
%


\section{Saturation}\label{aa}
The strategy in Fra{\v{n}}kov{\'a}'s extension of Helly's theorem may 
be abstracted as follows. Assume given a criterion for precompactness of  
sequences in some function space $\mathcal X$, which is contained in a larger 
function space $\mathcal Y$. One then considers sequences $(u_n)$ 
in $\mathcal Y$ 
with the property that for each $\veps>0$ there is a ``nearby'' sequence 
$(v_n^\veps)_n\subset \mathcal X$ which satisfies the original criterion.
Fra{\v{n}}kov{\'a}'s work \cite{fr}  shows that this strategy provides a 
genuine extension of Helly's theorem in $\mathcal X=\bpv(I)$
to a criterion for precompactness in $\mathcal Y=\mathcal R(I)$. 
Similarly, Theorem \ref{multi_var_frankova_any_p} above gives a proper 
extension of Theorem \ref{multi_var_compact} from $\mathcal X=\bv(\Omega)$ 
to $\mathcal Y=\mathcal R_p(\Omega)$ ($p\in[1,\infty]$).

It is natural to consider this strategy for other compactness criteria. 
In this section we consider two concrete cases: 
Fra{\v{n}}kov{\'a}'s own theorem and the Ascoli-Arzel\`{a} theorem.
However, our findings indicate that this approach 
does not readily generalize to other settings.

\subsection{Saturation of Fra{\v{n}}kov{\'a}'s theorem}\label{reg_sat}
Given Fra{\v{n}}kov{\'a}'s extension of Helly's theorem, it is natural to ask 
if one can ``take another step'' and obtain a further extension. 
The idea would be to have $\mathcal R(I)$ and $\evar$ play the roles of 
$\bv(I)$ and $\var$ in the proof of Theorem \ref{fr_thm} above. 
For $(u_n)\subset\mathcal B(I)$ we would thus 
make the following assumptions:
\begin{itemize}
	\item[(i)]  $(u_n)$ is bounded in $\mathcal B(I)$, and
	\item[(ii)] for each $\delta>0$ there is a sequence $(w_n^\delta)_n\subset\mathcal R(I)$
	satisfying $\|u_n-w_n^\delta\|\leq\delta$ for all $n\geq1$, and $(w_n^\delta)_n$ has
	uniformly bounded $\veps$-variations.
\end{itemize}
However, it is not hard to see that (ii) implies that  the original sequence $(u_n)$ itself
has uniformly bounded $\veps$-variations. 
Indeed, we may even weaken (ii) by requiring that for each $\delta>0$ there is a single (small) $\veps'>0$
such that the approximating sequence $(w_n^\delta)_n$ has uniformly bounded $\veps'$-variation:
\begin{itemize}
	\item[(ii)$'$] for each $\delta>0$ there is an $\veps'(\delta)>0$, with $\veps'(\delta)\to0$ as $\delta\to0$,
	and a sequence $(w_n^\delta)_n\subset\mathcal R(I)$
	satisfying $\|u_n-w_n^\delta\|\leq\delta$ for all $n\geq1$, and 
	\beq\label{weakened}
		\sup_n\,\,\veps'(\delta)\text{-var}\,w_n^\delta<\infty.
	\eeq
\end{itemize}
Assumption (ii)$'$ gives that, for each $\delta>0$, there is a $K_\delta<\infty$ such that 
$\veps'(\delta)\text{-var}\,w_n^\delta\leq K_\delta$ for all $n\geq1$. In particular,
for each $n\geq 1$ there is a function $z_n^\delta\in BPV(I)$ with 
$\|w_n^\delta-z_n^\delta\|\leq\veps'(\delta)$ and $\var z_n^\delta\leq K_\delta$.

Given $\veps>0$ we can then choose $\delta(\veps)>0$ so small that 
$\delta(\veps)+\veps'(\delta(\veps))\leq\veps$. Setting $M_\veps:=K_{\delta(\veps)}$ 
and $v_n^\veps:=z_n^{\delta(\veps)}$, we get
\[\|u_n-v_n^\veps\|\leq \|u_n-w_n^{\delta(\veps)}\|+\|w_n^{\delta(\veps)}-z_n^{\delta(\veps)}\|\leq \veps
\qquad\text{and}\qquad \var v_n^\veps\leq M_\veps,\]
for all $n\geq 1$. Thus, $\sup_n\evar u_n\leq M_\veps<\infty$ for each $\veps>0$, i.e., the original 
sequence $(u_n)$ has uniformly bounded $\veps$-variations.

This shows that the assumptions (i) and (ii)$'$ (a fortiori (i) and (ii)) imply that the given 
sequence already meets the assumptions of Theorem \ref{fr_thm}. That is,
Fra{\v{n}}kov{\'a}'s theorem is saturated in this respect.

\subsection{Saturation of the Ascoli-Arzel\`{a} theorem}\label{aa_sat}
It is natural to consider the same strategy for other standard compactness 
criteria. In this section we consider the Ascoli-Arzel\`{a} theorem, which can
be formulated as follows (see \cite{dib}, pp.\ 203-204). 
(For an open set $\Omega\subset\RR^N$ we equip $C(\Omega)$ with uniform norm $\|\cdot\|$.)

\begin{theorem}[Ascoli-Arzel\`{a}, version I] \label{aa1}
	Assume $(u_n)\subset C(\Omega)$ is a bounded sequence  for which there 
	is a modulus of continuity\footnote{I.e., $\omega:[0,\infty)\to[0,\infty)$ is continuous 
	and increasing and with $\omega(0)=0$.} $\omega$ such that
	\beq\label{mod_cont}
		\sup_n |u_n(x)-u_n(y)|\leq \omega(|x-y|)\qquad\text{for all $x,y\in\Omega$.}
	\eeq
	Then there is a subsequence $(u_{n(k)})\subset(u_n)$ and a function $u\in C(\Omega)$ 
	such that $u(x)=\lim_k u_{n(k)}(x)$ for every $x\in\Omega$.
\end{theorem}
\noindent (The convergence is uniform on compact subsets of $\Omega$, but this is 
not important for what follows.)
It is convenient to also record an alternative formulation in terms of uniform equicontinuity of $(u_n)$, 
i.e., the requirement that for any $\delta>0$ there is a $\eta(\delta)>0$ such that 
$\sup_n|u_n(x)-u_n(y)|\leq\delta$ whenever $x,y\in\Omega$ satisfy $|x-y|\leq \eta(\delta)$.
\begin{theorem}[Ascoli-Arzel\`{a}, version II]\label{aa2}
	If $(u_n)\subset C(\Omega)$ is bounded and uniformly equicontinuous, 
	then the  conclusion of Theorem \ref{aa1} holds.
\end{theorem}

To implement the extension strategy formulated above, we consider a 
bounded sequence $(u_n)\subset \mathcal B(\Omega)$ and assume that for each $\veps>0$
there is a $C(\Omega)$-sequence which is uniformly $\veps$-close 
to $(u_n)$ and which meets the conditions in Theorem \ref{aa1}. That is, for each 
$\veps>0$ there is a sequence $(v_n^\veps)_n\subset C(\Omega)$ and a modulus of continuity $\omega^\veps$
so that 
\beq\label{aa3}
	\|u_n-v_n^\veps\|\leq \veps\qquad\text{for all $n$,}
\eeq 
and
\beq\label{aa4}
	\sup_n |v_n^\veps(x)-v_n^\veps(y)|\leq \omega^\veps(|x-y|)
	\qquad\text{for all $x,y\in\Omega$.}
\eeq
However, these conditions imply that the 
original sequence $(u_n)$ itself meets the conditions in Theorem \ref{aa2}.
 Indeed, for any $n$ and for any $\veps>0$, \eq{aa3}-\eq{aa4} yield
\[|u_n(x)-u_n(y)|\leq |u_n(x)-v_n^\veps(x)|+|v_n^\veps(x)-v_n^\veps(y)|
+|v_n^\veps(y)-u_n(y)|\leq 2\veps+\omega^\veps(|x-y|).\]
Then, given $\delta>0$, set $\veps:=\frac{\delta}{3}$ and  choose $\eta(\delta)$ 
so small that $\omega^\veps(s)\leq\frac{\delta}{3}$ for $s\in[0,\eta(\delta)]$.
This gives 
\[\sup_n|u_n(x)-u_n(y)|\leq\delta\qquad\text{whenever $x,y\in\Omega$ satisfy $|x-y|\leq \eta(\delta)$.}\]
In particular, $(u_n)$ is a sequence in $C(\Omega)$
to which Theorem \ref{aa2} applies. This shows that also the Ascoli-Arzel\`{a} theorem is saturated
with respect to Fra{\v{n}}kov{\'a}'s extension strategy.

\begin{remark}
	Similar considerations show that the Kolmogorov-Riesz theorem characterizing
	precompactness in $L^p(\RR^N)$ is likewise saturated. 
\end{remark}



\begin{bibdiv}
\begin{biblist}
\bib{afp}{book}{
   author={Ambrosio, Luigi},
   author={Fusco, Nicola},
   author={Pallara, Diego},
   title={Functions of bounded variation and free discontinuity problems},
   series={Oxford Mathematical Monographs},
   publisher={The Clarendon Press, Oxford University Press, New York},
   date={2000},
   pages={xviii+434},
   isbn={0-19-850245-1},
   review={\MR{1857292}},
}
\bib{da}{article}{
   author={Davison, T. M. K.},
   title={A generalization of regulated functions},
   journal={Amer. Math. Monthly},
   volume={86},
   date={1979},
   number={3},
   pages={202--204},
   issn={0002-9890},
   review={\MR{522344}},
   doi={10.2307/2321523},
}
\bib{dib}{book}{
   author={DiBenedetto, Emmanuele},
   title={Real analysis},
   series={Birkh\"{a}user Advanced Texts: Basler Lehrb\"{u}cher. [Birkh\"{a}user
   Advanced Texts: Basel Textbooks]},
   publisher={Birkh\"{a}user Boston, Inc., Boston, MA},
   date={2002},
   pages={xxiv+485},
   isbn={0-8176-4231-5},
   review={\MR{1897317}},
   doi={10.1007/978-1-4612-0117-5},
}
\bib{di}{book}{
   author={Dieudonn\'{e}, J.},
   title={Foundations of modern analysis},
   series={Pure and Applied Mathematics, Vol. X},
   publisher={Academic Press, New York-London},
   date={1960},
   pages={xiv+361},
   review={\MR{0120319}},
}
\bib{dn}{book}{
   author={Dudley, Richard M.},
   author={Norvai\v{s}a, Rimas},
   title={Differentiability of six operators on nonsmooth functions and
   $p$-variation},
   series={Lecture Notes in Mathematics},
   volume={1703},
   note={With the collaboration of Jinghua Qian},
   publisher={Springer-Verlag, Berlin},
   date={1999},
   pages={viii+277},
   isbn={3-540-65975-7},
   review={\MR{1705318}},
   doi={10.1007/BFb0100744},
}
\bib{fr}{article}{
   author={Fra\v{n}kov\'{a}, Dana},
   title={Regulated functions},
   language={English, with Czech summary},
   journal={Math. Bohem.},
   volume={116},
   date={1991},
   number={1},
   pages={20--59},
   issn={0862-7959},
   review={\MR{1100424}},
}
\bib{le}{book}{
   author={Leoni, Giovanni},
   title={A first course in Sobolev spaces},
   series={Graduate Studies in Mathematics},
   volume={105},
   publisher={American Mathematical Society, Providence, RI},
   date={2009},
   pages={xvi+607},
   isbn={978-0-8218-4768-8},
   review={\MR{2527916}},
   doi={10.1090/gsm/105},
}
\end{biblist}
\end{bibdiv}
\end{document}